\documentclass[reqno, 11pt]{amsart}

\usepackage{amsthm,amsmath,amssymb,amsfonts,bm}
\usepackage{stmaryrd}
\usepackage{bbold}
\usepackage{times}
\usepackage[all]{xy}
\usepackage{mathrsfs}
\usepackage{graphicx}
\usepackage{tikz}
\usetikzlibrary{matrix,calc,arrows}
\usepackage{comment}

\theoremstyle{plain}
\newtheorem{thm}{Theorem}
\newtheorem{lem}[thm]{Lemma}
\newtheorem{cor}[thm]{Corollary}
\newtheorem{prop}[thm]{Proposition}

\theoremstyle{definition}

\newtheorem*{claim}{Claim}

\newtheorem*{notat}{Notations and conventions}

\theoremstyle{remark}

\newcommand{\Hom}{\operatorname{Hom}}

\newcommand{\ant}{{\mathsf{S}}}

\newcommand{\define}{\ \stackrel{\text{\rm def}}{=}\ }

\newcommand{\Id}{\operatorname{Id}}

\newcommand{\onto}{\twoheadrightarrow}

\newcommand{\gen}[1]{\langle{#1}\rangle}

\renewcommand{\k}{\mathbb{k}}
\renewcommand{\phi}{\varphi}

\newcommand{\bdot}{\,\text{\raisebox{-.45ex}{$\boldsymbol{\cdot}$}}\,}

\newcommand{\cC}{\mathcal{C}}

\renewcommand{\d}{\delta}

\newcommand{\fg}{\mathfrak{g}}

\newcommand{\fsl}{\mathfrak{sl}}

\newcommand{\e}{\varepsilon}

\newcommand{\ZZ}{\mathbb{Z}}

\newcommand{\cat}[1]{\operatorname{\mathsf{#1}}}
\newcommand{\Mod}[1]{{}_{#1}\!\cat{Mod}}

\newcommand{\fin}[1]{{#1}_{\text{\rm fin}}}
\newcommand{\ad}[1]{{#1}_{\text{\rm ad}}}
\newcommand{\adfin}[1]{{#1}_{\text{\rm ad\,fin}}}

\newdir{ >}{{}*!/-5pt/\dir{>}}

\begin{document}

\title[On the Adjoint Representation of a Hopf Algebra]%
{On the Adjoint Representation of a Hopf Algebra}

\author{Stefan Kolb}
\address{School of Mathematics, Statistics and Physics, Newcastle University, Newcastle
NE1 7RU, UK}

\author{Martin Lorenz}
\address{Department of Mathematics, Temple University, Philadelphia, PA 19122, USA}
    
\author{Bach Nguyen}
%\address{Department of Mathematics, Temple University, Philadelphia, PA 19122}
    
\author{Ramy Yammine}
%\address{Department of Mathematics, Temple University, Philadelphia, PA 19122}

%\email{lorenz@temple.edu}
%\urladdr{http://www.math.temple.edu/$\stackrel{\sim}{\phantom{.}}$lorenz}



\maketitle

\begin{abstract}
We consider the adjoint representation of a Hopf algebra $H$ focusing on the locally
finite part, $\adfin H$, defined as the sum of all finite-dimensional subrepresentations. 
For virtually cocommutative $H$ (i.e., $H$ is finitely generated as module over a cocommutative Hopf subalgebra), 
we show that $\adfin H$ is a Hopf subalgebra of
$H$. This is a consequence of the fact, proved here, that locally finite parts 
yield a tensor functor on the module category of 
any virtually pointed Hopf algebra. For general Hopf algebras, $\adfin H$ is shown to be a left coideal subalgebra.
We also prove a version of Dietzmann's Lemma from group theory for
Hopf algebras.
\end{abstract}

\maketitle

%%%%%%%%%%%%%%%%%%%%%%%%%%%%%%%%%%%%%%%%%%%%%%%%%%
% Intro
%%%%%%%%%%%%%%%%%%%%%%%%%%%%%%%%%%%%%%%%%%%%%%%%%%

\section{Introduction}

%%%%%%%%%%%%%%%%%%%%%%%%%%%%%%%%%%%%%%%%%%%%%%%%%%

\subsection{} 
\label{SS:Intro1}

Let $H$ be a Hopf algebra over a field $\k$, with comultiplication $\Delta h = h_{(1)} \otimes h_{(2)}$,
antipode $\ant$, and counit $\e$. The left \emph{adjoint action} of $H$ on itself is defined by
\begin{equation}
\label{E:ad}
k.h = k_{(1)} h\, \ant(k_{(2)}) \qquad (h,k \in H).
\end{equation}
%Here, $\Delta h = h_{(1)} \otimes h_{(2)}$ denotes the comultiplication of $H$ and $\ant$ is the antipode. 
This action makes $H$ a left $H$-module algebra that will be denoted by $\ad H$.
Our main interest is in the locally finite part,
\[
\adfin H = \{ h \in H \mid \dim_\k H.h < \infty \}.
\]
Of course, if $H$ is finite dimensional, then $\adfin H = H$; so we are primarily concerned with
infinite-dimensional Hopf algebras. In the case of a group algebra $H = \k G$, 
it is well known and easy to see that $\adfin H = \k\Delta$, the subgroup algebra of the so-called FC-center
$\Delta = \Delta(G) = \{ g \in G \mid g \text{ has finitely many $G$-conjugates}\}$. The $\Delta$-notation
for FC-centers, introduced by Passman and not to be confused with the 
traditional comultiplication notation, has led
to the nomenclature ``$\Delta$-methods'' in the study of group algebras:
a significant number of ring-theoretic properties of group algebras are controlled by the FC-center; 
see \cite{dP77}. 
The locally finite part of the adjoint representation has since also been explored for enveloping algebras \cite{jBdP90}, 
\cite{jBdP92}, \cite{jBdP93}
and for more general Hopf algebras \cite{jB06}, \cite{aJgL94}, \cite{gL02}.

%%%%%%%%%%%%%%%%%%%%%%%%%%%%%%%%%%%%%%%%%%%%%%%%%%

\subsection{}
\label{SS:Intro2}

%More generally, 
Locally finite parts may of course be defined for arbitrary representations as the sum of all finite-dimensional 
subrepresentations: for any left module $V$ over a $\k$-algebra $R$,
\[
\fin V \define \{ v \in V \mid \dim_\k R.v < \infty \}.
\]
This gives a functor $\fin\bdot$ on the category $\Mod R$ of left $R$-modules.
If $R$ is finitely generated as right module over some
subalgebra $T$, then $\fin V = \{ v \in V \mid \dim_\k T.v < \infty \}$ for any $V \in \Mod R$.
Adopting group-theoretical terminology, we will call a Hopf algebra $H$ \emph{virtually} of type $\cC$, where 
$\cC$ is a given class of Hopf algebras, if $H$ is finitely generated as right module over some
Hopf subalgebra $K \in \cC$. 
%(This notion is right-left symmetric if the antipode $\ant$ is surjective.)

%%%%%%%%%%%%%%%%%%%%%%%%%%%%%%%%%%%%%%%%%%%%%%%%%%

\subsection{}
\label{SS:Intro3}

The locally finite part $\fin A$ of any left $H$-module algebra $A$
is a subalgebra that contains the algebra of $H$-invariants,  $A^H = \{ a \in A \mid
h.a = \gen{\e,h}a \text{ for all } h \in H\}$. For $A = \ad H$, the invariant algebra 
coincides with the center of $H$ \cite[Lemma 10.1]{mL18}. The center is
rarely a Hopf subalgebra of $H$, even if $H$ is a group algebra,  
and $\adfin H$ need not be a Hopf subalgebra either in general, for example when $H$
is a quantized enveloping algebra; see \cite[Example 2.8]{jB06} or \cite{gL02}. 
However, we have the following result.
Recall that a left coideal subalgebra of $H$ is a subalgebra $C$ that is 
also a left coideal of $H$, i.e., $\Delta(C) \subseteq H \otimes C$.

\begin{thm}
\label{T:Cocom}
\begin{enumerate}
\item
$\adfin H$ is always a left coideal subalgebra of $H$.
\item
If $H$ is virtually cocommutative, then $\adfin H$ is a Hopf subalgebra of $H$.
\end{enumerate}
\end{thm}

For a quantized enveloping algebra 
of a complex semisimple Lie algebra, $H = U_q(\fg)$, part (a) is due to Joseph and Letzter: it follows from
\cite[Theorem 4.10]{aJgL94} that $\adfin{U_q(\fg)}$ is a left coideal subalgebra and 
this fact is also explicitly stated as \cite[Theorem 5.1]{gL02}.
Part (b) extends an earlier result of Bergen \cite[Theorem 2.18]{jB06} 
to arbitrary characteristics.
While Bergen's proof is based on his joint work with Passman \cite{jBdP93}, which determines
$\adfin H$ explicitly for
group algebras and for enveloping algebras of Lie algebras in characteristic $0$,
our approach uses general Hopf-theoretic methods.

%%%%%%%%%%%%%%%%%%%%%%%%%%%%%%%%%%%%%%%%%%%%%%%%%%

\subsection{}
\label{SS:Intro4}

Some of our work is in the context of (virtually) pointed Hopf algebras.
Over an algebraically closed field, all cocommutative Hopf algebras are pointed \cite[Lemma 8.0.1]{mS69}.
However, many pointed Hopf algebras of interest are not necessarily cocommutative.
Examples include the algebras of polynomial functions
of solvable connected affined algebraic groups (over an algebraically closed base field) and quantized
enveloping algebras of semisimple Lie algebras.
It turns out that $\fin\bdot$ is a tensor functor for virtually pointed Hopf algebras:

\begin{thm}
\label{T:LocFin}
If $H$ is virtually pointed, then $\fin{(V \otimes W)} = \fin V \otimes \fin W$  
for any $V,W \in \Mod H$.
\end{thm}

%%%%%%%%%%%%%%%%%%%%%%%%%%%%%%%%%%%%%%%%%%%%%%%%%%

\subsection{}
\label{SS:IntroDietzmann}

A standard group-theoretic fact, known as Dietzmann's Lemma, states that any finite subset
of a group that is stable under conjugation and consists of torsion elements generates a finite subgroup
(\cite{aD37} or \cite[\S 53]{aK56}). Our final result is the following version of Dietzmann's Lemma
for arbitrary Hopf algebras.

\begin{prop}
\label{P:Dietz}
Let $C_1,\dots,C_k$ be finite-dimensional left coideal subalgebras of $H$ and assume that
$C = \sum_{i=1}^k C_i$ is stable under the adjoint action of $H$. Then $C$ generates a finite-dimensional 
subalgebra of $H$.
\end{prop}

Of course,  the subalgebra that is generated by $C$ is also a left coideal subalgebra,
stable under the adjoint $H$-action, and it is contained in $\adfin H$.
Our proof of Proposition~\ref{P:Dietz} will show that if all $C_i$ are in fact sub-bialgebras of $H$,
then it suffices to assume that $C$ is stable under the adjoint actions of all $C_i$.

%%%%%%%%%%%%%%%%%%%%%%%%%%%%%%%%%%%%%%%%%%%%%%%%%%

\begin{notat}
We work over an arbitrary base field $\k$ and continue to write $\otimes = \otimes_\k$\,. 
As usual, $\bdot^* = \Hom_\k(\bdot\,,\k)$ denotes linear duals of $\k$-vector spaces.
Throughout, $H$ is a Hopf $\k$-algebra with
counit $\e$, antipode $\ant$, and comultiplication $\Delta h = h_{(1)} \otimes h_{(2)}$. 
\end{notat}

%%%%%%%%%%%%%%%%%%%%%%%%%%%%%%%%%%%%%%%%%%%%%%%%%%

\section{Proofs} 
\label{S:Proofs}

%%%%%%%%%%%%%%%%%%%%%%%%%%%%%%%%%%%%%%%%%%%%%%%%%%

\subsection{Proof of Theorem~\ref{T:Cocom}(a)}
\label{SS:ProofCocom(a)}

It follows from \eqref{E:ad} that the comultiplication satisfies
\begin{equation}
\label{E:D}
\Delta(k.h) = k_{(1)}h_{(1)}\ant(k_{(3)}) \otimes k_{(2)}.h_{(2)} \qquad (h,k \in H).
\end{equation}
For a given $a \in \adfin H$, consider the subspace $V = H.a$ and let $L$ denote the left coideal of $H$ that is generated
by $V$. Since $\dim_\k V < \infty$, we also have $\dim_\k L < \infty$ by the Finiteness
Theorem for comodules \cite[9.2.2]{mL18}. For any $h \in H$, we have 
$h_{(1)} \otimes h_{(2)}.a \otimes h_{(3)} \in H \otimes V \otimes H$.
Applying $\Id \otimes \Delta \otimes \Id$ and using \eqref{E:D}, we obtain
\[
h_{(1)} \otimes h_{(2)}a_{(1)}\ant(h_{(4)}) \otimes h_{(3)}.a_{(2)} \otimes h_{(5)} \in H \otimes H \otimes L \otimes H,
\]
which in turn implies
\[
\ant(h_{(1)}) h_{(2)}a_{(1)}\ant(h_{(4)})h_{(5)} \otimes h_{(3)}.a_{(2)} =  a_{(1)} \otimes h.a_{(2)} \in H \otimes L.
\]
Consequently, $\Delta a = a_{(1)} \otimes a_{(2)} \in H \otimes \adfin H$, 
proving Theorem~\ref{T:Cocom}(a). \qed

%%%%%%%%%%%%%%%%%%%%%%%%%%%%%%%%%%%%%%%%%%%%%%%%%%

\subsection{Tensors}
\label{SS:Tens}

Let $V$ and $W$ be left $H$-modules, with $H$-operations indicated by a dot, 
and view $V \otimes W$ as $H$-module with the
usual $H$-operation on tensors: $h.(v\otimes w) = h_{(1)}.v \otimes h_{(2)}.w$. 
For any $\k$-subspace $U \subseteq V \otimes W$, put
\begin{equation*}
%\label{E:U'U''}
U' = \sum_{f \in W^*} (\Id_V \otimes f)(U) \subseteq V \otimes \k = V
\end{equation*}
and \begin{equation*}
%\label{E:U'U''}
U'' = \sum_{f \in V^*} (f \otimes \Id_W)(U) \subseteq \k \otimes W = W.
\end{equation*}
Parts (a), (b) of the following lemma are standard (e.g., \cite[chap.~II \S7.8]{nB70}), but (c) may be new.

\begin{lem}[notation as above]
\label{L:U'U''}
\begin{enumerate}
\item
$U'$ is a $\k$-subspace of $V$ such that $U \subseteq U' \otimes W$; in fact, $U'$ is the smallest such subspace 
(contained in all others). Similarly for $U''$. Moreover, $U \subseteq U'\otimes U''$.
\item
$U$ is finite dimensional if and only if $U'$ and $U''$ are both finite dimensional.
\item
Assume that $H$ is pointed. If $U$ is an $H$-submodule of $V \otimes W$, then $U'$ and $U''$
are $H$-submodules of $V$ and $W$, respectively.
\end{enumerate}
\end{lem}

\begin{proof}
(a)
If $U \subseteq V' \otimes W$ for a subspace $V' \subseteq V$, then 
$(\Id_V \otimes f)(U) \subseteq (\Id_V \otimes f)(V' \otimes W) \subseteq V' \otimes \k = V'$ for all $f \in W^*$.
Thus, $U' \subseteq V'$, proving the minimality statement. On the other hand, any $u \in U$ can
be written as a finite sum $u = \sum_i v_i \otimes w_i$ with $v_i \in V$,  $w_i \in W$ and the $w_i$ may be chosen to 
be $\k$-linearly independent. Fixing $f_j \in W^*$ such that $\gen{f_j,w_i} = \d_{i,j}$, we obtain
$(\Id_V \otimes f_i)(u) = v_i \in U'$. This shows that $U \subseteq U' \otimes W$. Similarly, $U \subseteq V \otimes U''$
and so $U \subseteq (U' \otimes W) \cap (V \otimes U'') = U' \otimes U''$.

(b)
One direction is clear from the  inclusion $U \subseteq U'\otimes U''$. 
Now assume that $\dim_\k U = 1$, say $U = \k u$ with $u$ as in the proof of (a). Then $U'$ is generated
by the vectors $(\Id_V \otimes f)(u) = \sum_i v_i \gen{f,w_i}$ belonging to the subspace generated by 
the (finitely many) $v_i$. Thus $U'$ is finite dimensional in this case. Since $\bdot'$ evidently commutes 
with summation of subspaces,
it follows that $U'$ is finite dimensional whenever $U$ is so. The argument for $U''$ is analogous.

(c)
We will prove the following more general claim, for $H$ pointed. 

\begin{claim}
Let  $U \subseteq V \otimes W$ be a $\k$-subspace.
Then $H.U'' \subseteq U''$ if and only if $H.U \subseteq V \otimes U''$.
\end{claim}

One direction is clear: $H.U \subseteq H.(U' \otimes U'')
\subseteq H.U' \otimes H. U'' \subseteq V \otimes U''$ if $H.U'' \subseteq U''$.
The reverse implication will be proved below. 
Granting it for now, we may also conclude by symmetry that $H.U \subseteq U' \otimes W$ 
implies $H.U' \subseteq U'$. 
If $U$ is an $H$-submodule of $V \otimes W$, then $H.U \subseteq U \subseteq U' \otimes U''$ and hence 
$H.U' \subseteq U'$ and $H.U'' \subseteq U''$ by the Claim, proving (c).

To prove the remaining direction of the Claim, assume that $H.U \subseteq V \otimes U''$.
Fix a $\k$-basis $(v_i)_i$ of $V$ and write an arbitrary given $u \in U$ as
$u = \sum_i v_i \otimes w_i$ with $w_i = w_i(u) \in W$. 
As in the proof of (a), one sees that the vectors $w_i$ for
the various $u \in U$ generate the vector space $U''$. We need to show that $h.w_i \in U''$ for all $i$
and all $h\in H$. Let $(H_n)_{n \ge 0}$ denote the coradical filtration of $H$. Then $h \in H_n$
for some $n$. If $n = 0$, then we may assume that $h$ is grouplike. Thus, $h.u = \sum_i h.v_i \otimes h.w_i \in H.U
\subseteq V \otimes U''$. Since $h$ is invertible,
$(h.v_i)_i$ is a $\k$-basis of $V$ and it follows as in the proof of (a)
that $h.w_i \in U''$, as desired. Now let $n > 0$ and assume that $H_{n-1}.w_i \subseteq U''$ for all $i$.
By the Taft-Wilson Theorem (see \cite[Theorem 5.4.1]{sM93} or \cite[Section 4.3]{dR12}), we may assume that
$\Delta h = x \otimes h  + h \otimes y + \sum_j h'_j \otimes h''_j$ with $x,y \in H_0$ grouplike and $h'_j, h''_j \in H_{n-1}$\,.
Thus, 
\[
\sum_i x.v_i \otimes h.w_i + \sum_i h.v_i \otimes y.w_i + \sum_{i,j} h'_j.v_i \otimes h''_j.w_i = h.u
\in H.U \subseteq V \otimes U''.
\]
By induction, all $y.w_i, h''_j.w_i \in U''$.
Therefore, $\sum_i x.v_i \otimes h.w_i \in V \otimes U''$. Since $(x.v_i)_i$ is a $\k$-basis of $V$, 
we deduce once more that $h.w_i \in U''$ for all $i$, completing the proof.
\end{proof}

%%%%%%%%%%%%%%%%%%%%%%%%%%%%%%%%%%%%%%%%%%%%%%%%%%

\subsection{Proof of Theorem~\ref{T:LocFin}}
\label{SS:LocFin}

We may assume that $H$ is pointed (\S\ref{SS:Intro2}). For given $V,W \in \Mod H$, 
the inclusion $\fin{(V \otimes W)} \supseteq \fin V \otimes \fin W$ follows directly from the fact that 
$H.(v \otimes w) \subseteq H.v \otimes H.w$ for all $v \in V$ and $w\in W$.
For the reverse inclusion, let
$u \in \fin{(V \otimes W)}$ and put $U = H.u$. Then $U$ is a finite-dimensional $H$-submodule of $V \otimes W$, and so
Lemma~\ref{L:U'U''} gives that $U \subseteq  U' \otimes U''$ with $U' \subseteq V$ and $U'' \subseteq W$ 
being finite-dimensional $H$-submodules. Therefore, $U' \subseteq \fin V$ and $U'' \subseteq \fin W$,
proving the desired inclusion. \qed

%%%%%%%%%%%%%%%%%%%%%%%%%%%%%%%%%%%%%%%%%%%%%%%%%%

\subsection{Field extensions}
\label{SS:Field}

Let $R$ be a $\k$-algebra and let $F/\k$ be a field extension. Consider the $F$-algebra $R_F = R \otimes F$
and, for any $V \in \Mod R$, put $V_F = V \otimes F \in \Mod{R_F}$.

\begin{lem}[notation as above]
\label{L:Field}
$\fin{(V_F)} = (\fin V)_F$\,.
\end{lem}

\begin{proof}
The inclusion $\fin{(V_F)} \supseteq (\fin V)_F$ is evident. For the reverse inclusion, let $v \in \fin{(V_F)}$
be given and write $v = \sum_{i=1}^r v_i \otimes \lambda_i$, where $v_i \in V$ and the $\lambda_i$ are $\k$-linearly
independent elements of $F$. We need to show that all $v_i \in \fin V$ for all $i$. Suppose 
this fails for $i=1$, say. Replacing $v$ by $v \lambda_1^{-1}$, we may assume
that $\lambda_1 = 1$. Since $\dim_\k R.v_1 = \infty$, we may recursively construct a sequence in $R$ by
putting $r_0 = 1$ and choosing $r_n \in R$ so that $r_n.v_1 \notin \gen{r_j.v_i \mid  i \le r, j < n}_\k$\,.
Since $v \in \fin{(V_F)}$,  we can let $m$ denote the first index so that the family $(r_j.v)_{j=0}^m$ is 
$F$-linearly dependent, say $\sum_{j=0}^m  r_j.v\,\xi_j = 0$ with $\xi_j \in F$ not all $0$. Then $\xi_m \neq 0$
by minimality of $m$; so we may assume that $\xi_m = 1$.
Thus,
\[
%\begin{aligned}
0 =  r_m.v + \sum_{j=0}^{m-1}  r_j.v \,\xi_j
= \sum_{i=1}^r r_m.v_i \otimes \lambda_i + \sum_{i=1}^r \sum_{j=0}^{m-1} r_j.v_i \otimes \lambda_i \xi_j\,.
%\end{aligned}
\]
Choose a $\k$-linear projection $\pi \colon F \onto \k$ with $\pi(1) = 1$ but $\pi(\lambda_i) = 0$
for $i=2,\dots,r$ and apply the projection $\Id_V \otimes \pi \colon V_F \onto V$ to the above relation
to obtain
\[
0 = r_m.v_1 + \sum_{i=1}^r \sum_{j=0}^{m-1} r_j.v_i \,\pi( \lambda_i \xi_j).
\]
Thus, $r_m.v_1$ is a $\k$-linear combination of the vectors
$r_j.v_i$ $(i \le r, j < m)$, contrary to 
our construction of the sequence $(r_n)_{n\ge 0}$\,.
\end{proof}

We can now note the following consequence of Theorem~\ref{T:LocFin}.

\begin{cor}
\label{C:Field}
If $H$ is virtually cocommutative, then $\fin{(V \otimes W)} = \fin V \otimes \fin W$ for all
$V,W \in \Mod H$.
\end{cor}

\begin{proof}
Let $F$ denote an algebraic closure of $\k$. Then $H_F$ is virtually pointed and we may apply
Theorem~\ref{T:LocFin} and Lemma~\ref{L:Field} with $R = H$ to obtain
$\fin{(V \otimes W)} = \fin{(V_F \otimes_F W_F)} \cap (V \otimes W) 
= ((\fin V)_F \otimes_F (\fin W)_F) \cap (V \otimes W) = \fin V \otimes \fin W$. 
\end{proof}

%%%%%%%%%%%%%%%%%%%%%%%%%%%%%%%%%%%%%%%%%%%%%%%%%%

\subsection{Proof of Theorem~\ref{T:Cocom}(b)}
\label{SS:Cocom}
If $k \in K$ for some cocommutative Hopf subalgebra $K \subseteq H$, 
then Equation~\eqref{E:D} becomes 
\[
\Delta(k.h) = k_{(1)}h_{(1)}\ant(k_{(2)}) \otimes k_{(3)}.h_{(2)}
= k_{(1)}.h_{(1)} \otimes k_{(2)}.h_{(2)} = k.\Delta h.
\]
So $\Delta$ is a map in $\Mod K$. 
This also holds for the antipode $\ant \colon \ad H \to \ad H$
as is easily checked. Assuming $H$ to be finitely generated as right $K$-module,
locally finite parts can be calculated for $K$ (\S\ref{SS:Intro2}) and we obtain
$\Delta(\adfin H) \subseteq
\fin{(\ad H \otimes \ad H)} = \adfin H \otimes \adfin H$\,, where the last equality holds by Corollary~\ref{C:Field},
and $\ant(\adfin H) \subseteq \adfin H$.
Thus, $\adfin H$ is a Hopf subalgebra if $H$ is virtually cocommutative. \qed

%%%%%%%%%%%%%%%%%%%%%%%%%%%%%%%%%%%%%%%%%%%%%%%%%%

\subsection{Proof of Proposition~\ref{P:Dietz}}
\label{SS:ProofDietzmann}

For any $\k$-subspace $V \subseteq H$ and any $n \in \ZZ_+$, let $V^{(n)} \subseteq H$ denote the 
subspace that is generated by the products $v_1v_2\dots v_m$ with $v_i \in V$ and $m \le n$.
Thus, the subalgebra that is generated by $C = \sum_{i=1}^k C_i$ is equal to $\bigcup_{n \ge 0} C^{(n)}$.
We must show that this union is finite dimensional.
Since $C$ is finite dimensional, so are all $C^{(n)}$. Therefore, it suffices to show that
$C^{(k)} = C^{(k+i)}$ for all $i \ge 0$. To prove this, let $s > k$ and consider
a monomial of length $s$,
\[
x = c_{i_1}c_{i_2}\dots c_{i_s} \quad  (c_i \in C_i).
\]
We will show that $x \in C^{(s-1)}$, which will prove
the desired equality $C^{(s)} = C^{(s-1)}$.

Observe that not all indices $i_l$ in the above expression for $x$ can be distinct.
Let $\beta$ denote the shortest gap between any two factors in $x$ having the same index.
If $\beta = 0$, then two adjacent factors belong to the same $C_i$ and we may replace the pair by their
product in $C_i$, thereby representing $x$ as an element of $C^{(s-1)}$. Now assume that
$\beta > 0$ and, without loss, assume that a pair of factors with index $i=1$ has gap $\beta$.
Thus, $x$ contains a
length-$2$ submonomial of the form $cd$ with $c \in C_1$ being the first factor of the pair
and $d \in C_i$ $(i \neq 1)$.
Using the general formula $hk =  (h_{(1)}.k)h_{(2)}$ for $h,k \in H$
we may write $cd = (c_{(1)}.d)c_{(2)}$, a finite sum with all $c_{(1)}.d \in C$ and 
all $c_{(2)} \in C_1$, because $\Delta C_1 \subseteq H \otimes C_1$ and $H.C_i \subseteq C$. 
(If $C_1$ is a sub-bialgebra,
then it suffices to assume that $C_1.C \subseteq C$.) We may further
expand all $c_{(1)}.d  \in C$ into sums with terms from the various $C_i$.
Thus, with $y$ and $z$ denoting the initial and final segments of $x$
(possibly empty) before and after $cd$, respectively, the 
resulting sum for $x = ycdz = y(c_{(1)}.d)c_{(2)}z$ consists of length-$s$ monomials
of the same form as $x$ above, but having a lower $\beta$-value. We may therefore
argue by induction to finish the proof.

%%%%%%%%%%%%%%%%%%%%%%%%%%%%%%%%%%%%%%%%%%%%%%%%%%

\section{Remarks} 
\label{S:Remarks}

%%%%%%%%%%%%%%%%%%%%%%%%%%%%%%%%%%%%%%%%%%%%%%%%%%

\subsection{}
\label{SS:Virtual}

For the special case of cocommutative Hopf algebras, part (b) of Theorem~\ref{T:Cocom}
is a consequence of part (a), because left coideal subalgebras
then coincide with Hopf subalgebras. However, virtually cocommutative
Hopf algebras form a much wider class, which includes all finite-dimensional Hopf algebras.
Since Theorem~\ref{T:Cocom} is trivial in the finite-dimensional case, we mention the following
infinite-dimensional example. 

Let $n$ be odd and let $q\in \k$ be a primitive $n$-th root of unity. Let 
$U = U_q(\fsl_2)$ be the quantum enveloping algebra with standard generators $E, F, K^{\pm 1}$
as in \cite[I.3 and III.2]{kBkG02}.  
The elements $E^n$ and $F^n$ belong to the center of $U$ and $I = E^nU + F^nU$ is a Hopf ideal 
of $U$, because $\ant I \subseteq I$ and $\Delta E^n =E^n \otimes 1 + K^n \otimes E^n$,
$\Delta F^n = F^n \otimes K^{-n} + 1 \otimes F^n$ by the $q$-binomial formula; 
see \cite[Exercise 9.3.14]{mL18}. So $H = U/I$ is a Hopf algebra; it is not cocommutative
but virtually cocommutative, being finitely generated as (right and left) module over the 
cocommutative Hopf subalgebra $\k[K^{\pm 1}]$. In this example, $\adfin H = H$, 
because the adjoint (conjugation) action of $K$ on $H$ is locally finite.

%%%%%%%%%%%%%%%%%%%%%%%%%%%%%%%%%%%%%%%%%%%%%%%%%%

\subsection{}
\label{SS:Free}

In general, $H$ need not be free over $\adfin H$.
To see this, we recall the following result of Masuoka \cite{aM91}.
Assume that $H$ is pointed and let $B$ be a left coideal subalgebra of $H$. Then $H$ is free as a right or left 
$B$-module if and only if $\ant(B\cap \k G) = B\cap \k G$, where 
$G = G(H)$ is the group of grouplike elements.
This condition is often easy to check. For example, taking $H = U_q(\fg)$, $B = \adfin H$ 
and using the notation of \cite[Theorem 4.10]{aJgL94}, we have $\tau(\lambda)\in \adfin H \cap G(H)$ but 
$\ant(\tau(\lambda)) = \tau(-\lambda)\notin \adfin H$ for $\lambda\in -R^+(\pi)$. So $H$ is not free over $\adfin H$
in this case.

%%%%%%%%%%%%%%%%%%%%%%%%%%%%%%%%%%%%%%%%%%%%%%%%%%
%% Bibliography
%%%%%%%%%%%%%%%%%%%%%%%%%%%%%%%%%%%%%%%%%%%%%%%%%%

\def\cprime{$'$}
\providecommand{\bysame}{\leavevmode\hbox to3em{\hrulefill}\thinspace}
\providecommand{\MR}{\relax\ifhmode\unskip\space\fi MR }
\providecommand{\MRhref}[2]{%
  \href{http://www.ams.org/mathscinet-getitem?mr=#1}{#2}
}
\providecommand{\href}[2]{#2}

%%%%%%%%%%%%%%%%%%%%%%%%%%%%%%%%%%%%%%%%%%%%%%%%%%

\end{document}